\documentclass{birkjour}
\usepackage{amsthm}
\usepackage{amsfonts}
\usepackage{mathrsfs}
\usepackage{amscd,xy,amssymb,amsmath,graphicx,verbatim}

\usepackage[TS1,OT1,T1]{fontenc}
\newtheorem{theorem}{Theorem}[section]
\newtheorem{lemma}[theorem]{Lemma}
\newtheorem{corollary}[theorem]{Corollary}
\newtheorem{proposition}[theorem]{Proposition}

\newtheorem{definition}[theorem]{Definition}

\theoremstyle{remark}

\numberwithin{equation}{section}

\begin{document}

%
%
%
%
%
%
%
%
%

\title{Li-Yorke chaos translation set for linear operators}

\author[Bingzhe Hou]{Bingzhe Hou$^\alpha$}

\address{School of Mathematics , Jilin university, 130012, Changchun, P. R. China}

\email{houbz@jlu.edu.cn}

\thanks{$^\alpha$The first author was supported by the National Nature
Science Foundation of China (Grant No. 11001099)}

\author[Lvlin Luo]{Lvlin Luo$^{\beta*}$}

\address{School of Mathematics and Statistics, Xidian University, 710071, Xi'an, P. R. China}
\address{School of Mathematical Sciences, Fudan University, 200433, Shanghai, P. R. China}
\email{luoll12@mails.jlu.edu.cn}
\email{luolvlin@fudan.edu.cn}
\thanks{\hspace*{-0.8cm}$^\beta$The second author was supported by the National Natural Science Foundation of China (Grant No. 11626180).}
\thanks{\hspace*{-0.8cm}$^*$Corresponding author}



\begin{abstract}
In order to study Li-Yorke chaos by the scalar perturbation for a given bounded linear operator $T$ on Banach spaces $X$, we introduce the Li-Yorke chaos translation set of $T$, which is defined by
$S_{LY}(T)=\{\lambda\in\mathbb{C};\lambda+T \text{ is Li-Yorke chaotic}\}$.
In this paper, some operator classes are considered, such as normal operator, compact operator, shift and so on. In particular, we show that the Li-Yorke chaos translation set of Kalisch operator on Hilbert space $\mathcal{L}^2[0,2\pi]$ is a simple point set $\{0\}$.
\end{abstract}

\maketitle

\section{Introduction}

Let $Y$ be a metric space and $f: Y\rightarrow Y$ be a continuous map. Then the pair $(Y,f)$ is called a topological dynamic system, which is induced by the iteration
$$f^n=\underbrace{f\circ\cdots\circ f}\limits_{n},\qquad n\in\mathbb{N}.$$
Chaos theory is an interesting topic in the research of dynamical systems. And several kinds of chaos have been raised and well studied, such as Li-Yorke chaos, Devaney chaos, distributional chaos and so on. In particular, Li-Yorke chaos is firstly strictly defined by Li and Yorke in 1975 \cite{LY}.
\begin{definition}
Let $(Y,f)$ be a topological dynamic system. $\{x,y\}\subseteq Y$ is said to be a Li-Yorke chaotic pair, if
\begin{equation*}
\limsup\limits_{n\rightarrow+\infty}d(f^{n}(x),f^{n}(y))>0 \ \ and \ \
\liminf\limits_{n\rightarrow+\infty}d(f^{n}(x),f^{n}(y))=0.
\end{equation*}
A subset $\Gamma\subseteq Y$ is called a scrambled set, if each pair of two
distinct points in $\Gamma$ is a Li-Yorke chaotic pair. Furthermore, $f$ is called (densely) Li-Yorke chaotic, if there exists an
uncountable (densely) scrambled set.
\end{definition}

A topological dynamic system $(X,T)$ is said to be linear, if both the space $X$ and the map $T$ are linear. The research of chaos in linear dynamical system began with the work of G. Godefroy and J. Shapiro on Devaney chaotic operator in \cite{GS}. So far, chaotic operator have been extensively studied (we refer to the books \cite{Bay} and \cite{GEP}). In this paper, we focus on the Li-Yorke chaos in linear dynamical system in which T. Berm\'{u}dez et al. and N. Bernardes Jr. et al. had done excellent works in \cite{Ber} and \cite{Bern}. Especially, there are several equivalent descriptions of Li-Yorke chaos in linear dynamical system.

Throughout this paper, $\mathbb{T}=\{z\in\mathbb{C},|z|=1\}$,
$\mathbb{D}=\{z\in\mathbb{C},|z|<1\}$,
$X$ denotes an infinite-dimensional separable complex Banach space,
and $\mathcal{H}$ denotes a separable complex Hilbert space.
Moreover, $\mathcal{B}(X)$ denotes the set of all bounded linear operator on $X$,
and $\mathcal{B}(\mathcal{H})$ denotes the set of all bounded linear operator on $\mathcal{H}$.

\begin{definition}\label{irreg}
Given an operator $T\in\mathcal{B}(X)$ and a vector $x\in X$.
We say that $x$ is a semi-irregular vector for $T$, if the sequence $\{T^{n}x\}_{n\in\mathbb{N}}$ does not converge to zero but it has a subsequence converging to zero. We say that $x$ is an irregular vector for $T$, if the sequence $\{T^{n}x\}_{n\in\mathbb{N}}$ is unbounded but it has a subsequence converging to zero.
\end{definition}

\begin{theorem}\label{LYequiv}(\cite{Ber, Bern})
If $T\in\mathcal{B}(X)$, then the following assertions are equivalent.

(i) $T$ is Li-Yorke chaotic;

(ii) $T$ admits a Li-Yorke chaotic pair;

(iii) $T$ admits a semi-irregular vector;

(iv) $T$ admits an irregular vector.
\end{theorem}

\begin{definition}\label{LYC}(\cite{Bern})
Let $T\in\mathcal{B}(X)$. We say that $T$ satisfies the Li-Yorke Chaos Criterion if there exists a subset $X_0$ of $X$ with the following properties.

(a) $(T^{n}x)$ has a subsequence converging to zero, for every $x\in X_0$.

(b) There is a bounded sequence $(a_n)$ in $\overline{{\rm span}(X_0)}$ such that the sequence $(T^{n}a_n)$ is unbounded.
\end{definition}
\begin{theorem}\label{LYequivC}\cite{Ber, Bern}
If $T\in\mathcal{B}(X)$, then the following assertions are equivalent.

(i) $T$ is Li-Yorke chaotic;

(ii) $T$ satisfies the Li-Yorke Chaos Criterion.
\end{theorem}

Let $T\in\mathcal{B}(X)$, our aim is to study the Li-Yorke chaos of $T$ plusing a scalar. So we introduce the concept of Li-Yorke chaos translation set.
\begin{definition}\label{lycts}
For $T\in\mathcal{B}(X)$,
the Li-Yorke chaos translation set of $T$, denoted by $S_{LY}(T)$, is defined by
$$
S_{LY}(T)=\{\lambda\in\mathbb{C};\lambda+T \text{ is Li-Yorke chaotic}\}.
$$
\end{definition}
Replacing Li-Yorke chaos by some other dynamical property $\mathfrak{P}$, we can also define $\mathfrak{P}$ translation set of $T$, such as distributional chaos, hypercyclicity, frequent hypercyclicity and so on. In particular, if $\mathfrak{P}$ means  non-invertibility, the  non-invertibility translation set of $T$ coincide with the spectrum of $T$.

In this paper, we discuss the Li-Yorke chaos translation set for a given bounded linear operator. Some operator classes are considered, such as normal operator, compact operator, shift and so on. In particular, we show that the Li-Yorke chaos translation set of Kalisch operator on Hilbert space $\mathcal{L}^2[0,2\pi]$ is a simple point set $\{0\}$.

\section{Basic properties and some examples}

In this section, we will study the properties of the Li-Yorke chaos translation set and show some examples.
\begin{proposition}\label{basicp1}
If $T\in\mathcal{B}(X)$,
then $S_{LY}(T)$ is bounded.
\end{proposition}
\begin{proposition}\label{basicp2}
If $T_1,T_2\in\mathcal{B}(X)$,
then
$$
S_{LY}(T_1\oplus T_2)=S_{LY}(T_1)\cup S_{LY}(T_2).
$$
\end{proposition}

The above two properties of Li-Yorke chaos translation set of bounded linear operator are trivial, and are the same as the ones of spectrum. However, we will show that there are several differences between them. This is the reason of that we use the terminology "Li-Yorke chaos translation set" but not "Li-Yorke chaos spectrum". One difference is as follows.
\begin{proposition}\label{basicp3}
Let $T\in\mathcal{B}(X)$.
Then for any $\lambda\in \mathbb{C}$,
$$
S_{LY}(\lambda+T)=S_{LY}(T)-\lambda.
$$
\end{proposition}
Another difference is that it is possible to be empty of the Li-Yorke chaos translation set,
but the spectrum of a bounded linear operator is always nonempty,
which is a direct result of normal operators being not Li-Yorke chaotic (\cite{HTS, Ber}).
\begin{proposition}\label{normal}
If $T\in\mathcal{B}(\mathcal{H})$ is a normal operator, then $S_{LY}(T)=\emptyset$.
\end{proposition}
\begin{proof}
Since $T$ is a normal operator, then for every $\lambda\in \mathbb{C}$,
$\lambda+T$ is also a normal operator and hence is not Li-Yorke chaotic. Thus, $S_{LY}(T)=\emptyset$.
\end{proof}

In \cite{Hbz}, the authors studied chaos for a family of
operators given by Cowen and Douglas \cite{Cowen}.
\begin{definition}
For $\Omega$ a connected open subset of $\mathbb{C}$ and $n$ a
positive integer, let $B_{n}(\Omega)$ denotes the
operators $T$ in $\mathcal {B}(\mathcal {H})$ which satisfy:

(a) $\Omega \subseteq \sigma(T)=\{\omega \in \mathbb{C}: T-\omega
 {\ not \ invertible}\}$;

(b) ${\rm Ran}(T-\omega)=\mathcal {H} \ for \ \omega \ in \ \Omega$;

(c) $\bigvee {\rm Ker}_{\omega\in \Omega}(T-\omega)=\mathcal {H}$; and

(d) ${\rm dim} \ {\rm Ker}(T-\omega)=n$ for $\omega$ in $\Omega$.
\end{definition}
Obviously, hypercyclicity implies Li-Yorke chaos.
\begin{lemma}\label{cd}(\cite{Hbz})
Let $T\in B_n(\Omega)$. If $\Omega\cap\mathbb{T}\neq\emptyset$, then $T$ is Li-Yorke chaotic (indeed, it is mixing and Devaney chaotic).
\end{lemma}

N. Bernardes Jr. et al. gave a necessary condition for densely Li-Yorke chaos in \cite{Bern}. Densely Li-Yorke chaos is strictly stronger than Li-Yorke chaos. However, inspired by this result, we give a necessary condition for Li-Yorke chaos.
\begin{lemma}\label{ndLY}(\cite{Bern})
Let $T\in\mathcal{B}(X)$.
If $T^*$ has an eigenvalue $\gamma$ with $|\gamma|\geq1$, then $T$ is not densely Li-Yorke chaotic.
\end{lemma}

\begin{proposition}\label{nLY}
Let $T\in\mathcal{B}(X)$.
If
$$
X^*=\overline{{\rm span}\{\varphi\in X^*; \ T^*\varphi=\gamma\varphi \ for \ some \ \gamma\in\mathbb{C} \ with \ |\gamma|\geq1\}},
$$
then $T$ is not Li-Yorke chaotic.
\end{proposition}
\begin{proof}
Given any $0\neq x\in X$. Since
$$
X^*=\overline{{\rm span}\{\varphi\in X^*; \ T^*\varphi=\gamma\varphi \ for \ some \ \gamma\in\mathbb{C} \ with \ |\gamma|\geq1\}},
$$
there exists $\varphi\in X^*$ such that $T^*\varphi=\gamma \varphi$ with $|\gamma|\geq1$ and $\varphi(x)\neq0$. Then
$$
|\varphi(T^nx)|=|(T^*)^n\varphi(x)|=|\gamma^n\varphi(x)|\geq|\varphi(x)|.
$$
This implies that there is no subsequence $\{n_k\}_{k=1}^{+\infty}$ such that $\lim\nolimits_{k\rightarrow+\infty}T^n x=0$, and hence $T$ is not Li-Yorke chaotic.
\end{proof}
Similarly, one can see the following result.
\begin{proposition}\label{exp}
Let $T\in\mathcal{B}(X)$.
If
$$
X^*=\overline{{\rm span}\{\varphi\in X^*; \ T^*\varphi=\gamma\varphi \ for \ some \ \gamma\in\mathbb{C} \ with \ |\gamma|>1\}},
$$
then for any $0\neq x\in X$, $\lim\nolimits_{n\rightarrow+\infty}T^nx=\infty$.
\end{proposition}

\begin{corollary}\label{cdex}
Let $T\in B_n(\Omega)$. If $\Omega-\mathbb{D}\neq\emptyset$, then for any $0\neq x\in X$, $\lim\limits_{n\rightarrow+\infty}(T^*)^nx=\infty$ and hence $T^*$ is not Li-Yorke chaotic.
\end{corollary}

\begin{proposition}\label{shift}
Let $T$ be the unilateral forward shift operator on $l^2(\mathbb{N})$ (the space of square summable sequences),
$$
T(x_1,x_2,\cdots)=(0,x_1,x_2,\cdots), \ \ \ \ for \ any \ (x_1,x_2,\cdots)\in l^2(\mathbb{N}).
$$
Then $S_{LY}(T^*)=2\mathbb{D}\setminus\{0\}$,
$S_{LY}(2T^*)=3\mathbb{D}$ and $S_{LY}(T)=S_{LY}(2T)=\emptyset$.
\end{proposition}
\begin{proof}
For any $\lambda\in2\mathbb{D}\setminus\{0\}$, one can see $(\lambda+\mathbb{D})\bigcap\mathbb{T}\neq\emptyset$. Together with $\lambda+T^*\in B_1(\lambda+\mathbb{D})$, we have $\lambda+T^*$ is
Li-Yorke chaotic by Lemma \ref{cd}. For any $\lambda\notin2\mathbb{D}$,
$$
\|\lambda+T^*\|\geq \|\lambda\|-\|T^*\|\geq2-1=1,
$$
then $\lambda+T^*$ is not Li-Yorke chaotic. Obviously, $T^*$ is not Li-Yorke chaotic. Therefore, $S_{LY}(T^*)=2\mathbb{D}\setminus\{0\}$. As well-known, $2T^*$ is Li-Yorke chaotic. Then, in a similar manner, one can obtain $S_{LY}(2T^*)=3\mathbb{D}$.

For any $\lambda\in\mathbb{C}\setminus\{0\}$, $(\lambda+T)^*=\overline{\lambda}+T^*$ is a Cowen-Douglas operator, more precisely, $(\lambda+T)^*\in B_1(\overline{\lambda}+\mathbb{D})$. Then $\overline{\lambda}+T^*$ satisfies the condition in Corollary \ref{cdex}, and consequently $\lambda+T$ is not Li-Yorke chaotic. It is obvious that $T$ is an isometry and hence is not Li-Yorke chaotic.  Therefore, $S_{LY}(T)=\emptyset$. Similarly, $S_{LY}(2T)=\emptyset$.
\end{proof}

Now, one can get that spectrum of a given bounded linear operator must be closed but the Li-Yorke chaos translation set of which maybe open. Next, let us consider compact operator. From works of \cite{HTS, Ber, HTZ, Bern}, one can get the following two results.
\begin{lemma}\label{nc}(\cite{HTS, Ber, Bern})
Let $T\in\mathcal{B}(\mathcal{H})$ be a compact operator. Then $T$ is not Li-Yorke chaotic.
\end{lemma}
\begin{lemma}\label{spec}(\cite{Ber, HTZ})
Let $T\in\mathcal{B}(\mathcal{H})$.
If $T$ is Li-Yorke chaotic, then $\sigma(T)\cap\mathbb{T}\neq\emptyset$.
\end{lemma}
\begin{proposition}\label{compact}
Let $T\in\mathcal{B}(X)$ be a compact operator. Then $S_{LY}(T)\subseteq\mathbb{T}$.
\end{proposition}
\begin{proof}
Given any $\lambda\notin\mathbb{T}$. According to Riesz's decomposition theorem, we have direct sum decompositions $\lambda+T=T_1\oplus T_2$ and $X=X_1\oplus X_2$,
$$
\lambda+T=\begin{matrix}\begin{bmatrix}
T_{1},&0\\
0, &T_{2}\\
\end{bmatrix}&
\begin{matrix}
X_1\\
X_2\end{matrix}\end{matrix}
$$
where $\sigma(T_1)\cap\mathbb{T}=\emptyset$ and $X_2$ is a finite dimensional vector space. Then by Lemma \ref{nc} and \ref{spec}, neither $T_1$ nor $T_2$ is Li-Yorke chaotic, and consequently $\lambda+T$ is not Li-Yorke chaotic. Thus, $S_{LY}(T)\subseteq\mathbb{T}$.
\end{proof}

To study the Li-Yorke chaos translation set of compact quasinilpotent shift, the following conclusion is useful.

\begin{lemma}\label{sa}(\cite{Salas, Bay})
Let $T\in\mathcal{B}(X)$.
If ${\rm span}(\bigcup_{n\in\mathbb{N}}{\rm Ker}(B^n)\cap {\rm Ran}(B^n))$ is dense in $X$, then $I+T$ is mixing and hence is Li-Yorke chaotic.
\end{lemma}

\begin{proposition}\label{ncshift}
Let $T$ be a compact quasinilpotent unilateral forward weighted shift operator on $l^{2}(\mathbb{N})$ with weight
sequence $\{\omega_n\}_{n=1}^{+\infty}$,
$$
T(x_1,x_2,\cdots)=(0,\omega_1x_1,\omega_2x_2,\cdots) \ \ \ \ for \ any \ (x_1,x_2,\cdots)\in l^2(\mathbb{N}),
$$
where $\omega_n\neq0$ for all $n\in \mathbb{N}$ and $\omega_n\rightarrow0$ as $n\rightarrow+\infty$. Then $S_{LY}(T^{*})=\mathbb{T}$ and $S_{LY}(T)=\emptyset$.
\end{proposition}
\begin{proof}
Since $T^{*}$ is a compact operator, it follows from Proposition \ref{compact} that $S_{LY}(T^{*})\subseteq\mathbb{T}$. Notice that for any $\lambda\in\mathbb{T}$,
$\lambda+T^{*}$ is Li-Yorke chaotic if and only if $1+\overline{\lambda} T^{*}$ is Li-Yorke chaotic. Then by Proposition \ref{sa}, $\lambda+T^{*}$ is Li-Yorke chaotic for any given $\lambda\in\mathbb{T}$. Therefore, $S_{LY}(T^{*})=\mathbb{T}$.

On the other hand, since $T$ is a compact operator, it follows from Proposition \ref{compact} that $S_{LY}(T)\subseteq\mathbb{T}$.
Given any $0\neq x=(x_1,x_2,\cdots)\in l^2(\mathbb{N})$. Let $x_k$ be the first non-zero coordinate of $x$. Then for any $\lambda\in\mathbb{T}$ and any $n\in \mathbb{N}$,
$$
\|T^nx\|\geq\|\lambda^nx_k\|=\|x_k\|.
$$
This implies that $\lambda+T$ is not Li-Yorke chaotic. Hence, $S_{LY}(T)=\emptyset$.
\end{proof}
According to Proposition \ref{basicp2}, Proposition \ref{shift} and Proposition \ref{ncshift}, one can get that the Li-Yorke chaos translation set of bounded linear operator could be open, or closed, or neither open nor closed.

\section{Kalisch operator}

First, let us review the Kalisch operator, it is defined on $\mathcal{L}^2[0,2\pi]$ by
$$
Sf(\theta)=e^{\mathbf{i}\theta}f(\theta)-\int\nolimits_{0}^{\theta}\mathbf{i}e^{\mathbf{i}t}f(t)dt.
$$

This operator was introduced by G. K. Kalisch in \cite{Kal}, which admits $\sigma(S)=\sigma_p(S)=\mathbb{T}$. For any $\alpha\in[0,2\pi]$, let ${\bf 1}_{[\alpha,2\pi]}$ be the function valued 1 on $[\alpha,2\pi]$ and valued 0 else. Then ${\bf 1}_{[\alpha,2\pi]}$ is an eigenvector of $S$ associated with eigenvalue $e^{\mathbf{i}\alpha}$, and the set
$\{{\bf 1}_{[\alpha,2\pi]};\alpha\in[0,2\pi]\}$ span the whole space $\mathcal{L}^2[0,2\pi]$. The Kalisch-type operators were used to study frequent hypercyclicity, Devaney chaos and so on \cite{Bad, Bay}. In this section, we will show that the Li-Yorke chaos translation set of Kalisch operator is a simple point set $\{0\}$.

\begin{theorem}\label{Kalisch}
$S_{LY}(S)=\{0\}$.
\end{theorem}
\begin{proof}
First of all, we will show an identity of the iterations of the Kalisch operator. Given any $w\in \mathbb{C}$, $n\in \mathbb{N}$, and $f\in\mathcal{L}^2[0,2\pi]$. One can see for any $\theta\in[0,2\pi]$,
\begin{equation}\label{tn}
(w+S)^nf(\theta)=(w+e^{\mathbf{i}\theta})^nf(\theta)-n\int\nolimits_{0}^{\theta}\mathbf{i}e^{\mathbf{i}t}(w+e^{\mathbf{i}t})^{n-1}f(t)dt.
\end{equation}

Obviously, the equation holds for $n=1$. Suppose that it holds for $n=k$. Then for $n=k+1$ there is
\begin{align*}
&(w+S)^{k+1}f(\theta) \\
=&(w+S)^{k}(w+S)f(\theta) \\
=&(w+e^{\mathbf{i}\theta})^{k+1}f(\theta)-(w+e^{\mathbf{i}\theta})^{k}\int\nolimits_{0}^{\theta}\mathbf{i}e^{\mathbf{i}t}f(t)dt
-k\int\nolimits_{0}^{\theta}\mathbf{i}e^{\mathbf{i}t}(w+e^{\mathbf{i}t})^kf(t)dt \\
&+k\int\nolimits_{0}^{\theta}\mathbf{i}e^{\mathbf{i}t}(w+e^{\mathbf{i}t})^{k-1}(\int\nolimits_{0}^{t}\mathbf{i}e^{\mathbf{i}x}f(x)dx)dt \\
=&(w+e^{\mathbf{i}\theta})^{k+1}f(\theta)-(w+e^{\mathbf{i}\theta})^{k}\int\nolimits_{0}^{\theta}\mathbf{i}e^{\mathbf{i}t}f(t)dt
-k\int\nolimits_{0}^{\theta}\mathbf{i}e^{\mathbf{i}t}(w+e^{\mathbf{i}t})^kf(t)dt \\
&+(w+e^{\mathbf{i}\theta})^{k}\int\nolimits_{0}^{\theta}\mathbf{i}e^{\mathbf{i}x}f(x)dx-\int\nolimits_{0}^{\theta}\mathbf{i}e^{\mathbf{i}t}(w+e^{\mathbf{i}t})^kf(t)dt \\
=&(w+e^{\mathbf{i}\theta})^{k+1}f(\theta)-(k+1)\int\nolimits_{0}^{\theta}\mathbf{i}e^{\mathbf{i}t}(w+e^{\mathbf{i}t})^{k}f(t)dt.
\end{align*}
By the mathematical induction, the equation (\ref{tn}) holds for all $n\in \mathbb{N}$.

As well-known, the spectrum of $S$ is the unit circle. Then for $0\neq w\in \mathbb{C}$, the position relationship between the spectrum of $w+S$ and the unit circle must be one of the following cases (Fig. 1).

\hspace*{-2cm}\scalebox{0.33}[0.36]{\includegraphics{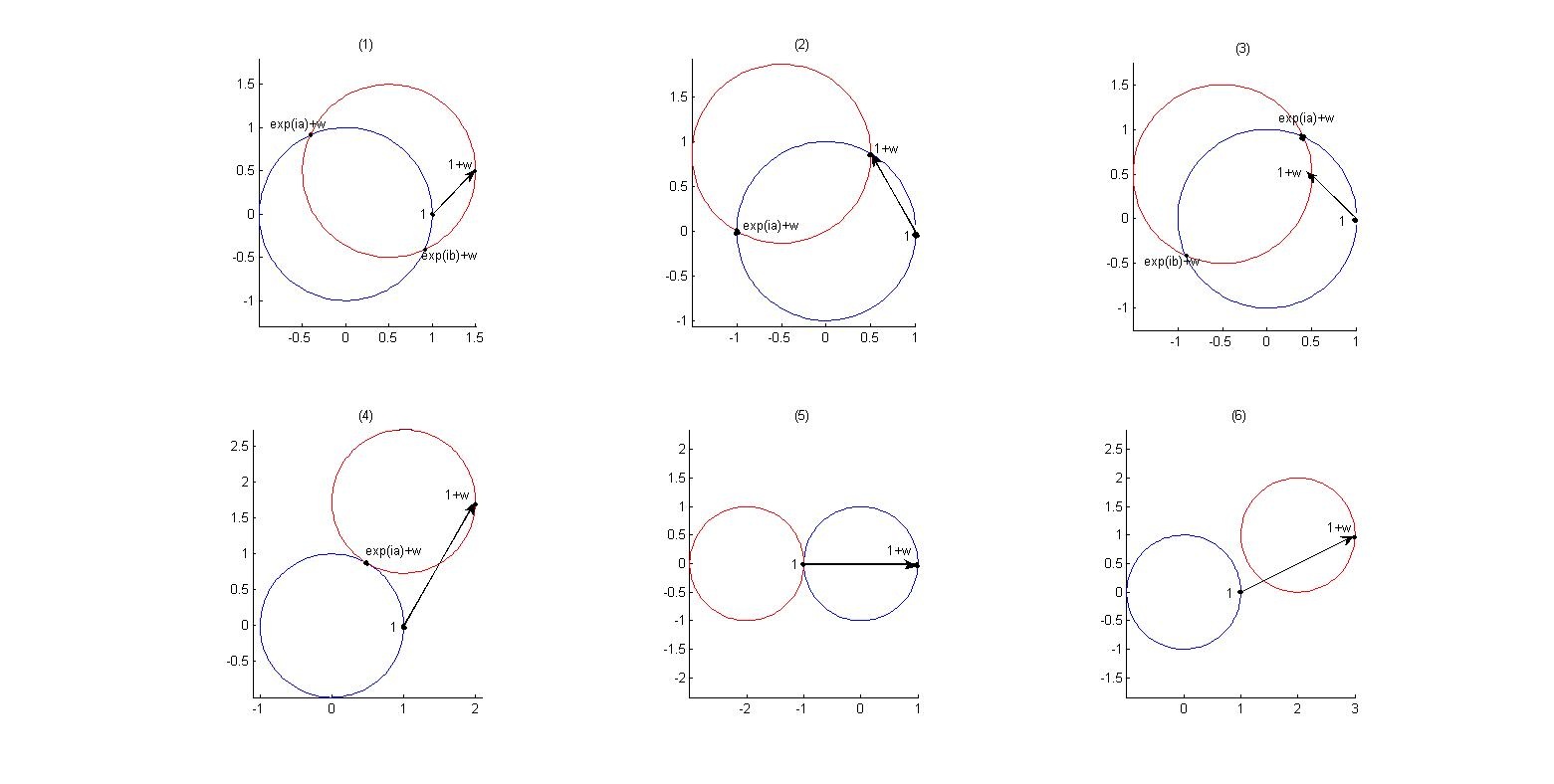}}
\begin{center}
{\small {\bf Fig. 1} }
\end{center}

It is easy to see that $w+S$ is not Li-Yorke chaotic if case (6) in Fig. 1 holds, since its spectrum is disjoint from the unit circle. We should prove that $w+S$ is not Li-Yorke chaotic when cases (1-5) hold, respectively. The methods are similar, then we will give details for case (1) as the example. Now, suppose case (1) holds. That means that $1+w$ is out of the unit disc, and the spectrum of $w+S$ intersects the unit circle at two distinct points, denoted by $w+e^{\mathbf{i}a}$ and $w+e^{\mathbf{i}b}$, $0<a<b<2\pi$. Let
$$
H_1=\{f\in\mathcal{L}^2[0,2\pi]; f(\theta)=0, \ for \ all \ \theta\in[a,2\pi]\},
$$
$$
H_2=\{f\in\mathcal{L}^2[0,2\pi]; f(\theta)=0, \ for \ all \ \theta\in[0,a]\cup[b,2\pi]\},
$$
$$
H_3=\{f\in\mathcal{L}^2[0,2\pi]; f(\theta)=0, \ for \ all \ \theta\in[0,b]\},
$$
and $P_i:\mathcal{L}^2[0,2\pi]\rightarrow H_i$ be the orthogonal projections, for $i=1,2,3$. Now, let us discuss the iterations of functions in $H_i$, respectively.

{\bf Claim 1.} For any $0\neq f\in H_1$, $\lim\nolimits_{n\rightarrow+\infty}P_1(w+S)^nf=\infty$.

Since $0\neq f\in H_1$, then there exists $s$ with $0<s<a$, such that $f(\theta)\neq0$ on a positive measure subset of $[0,s]$. Let
$$
H_0=\{f\in\mathcal{L}^2[0,2\pi]; f(\theta)=0, \ for \ all \ \theta\in[s,2\pi]\},
$$
and $P_0:\mathcal{L}^2[0,2\pi]\rightarrow H_0$ be the orthogonal projection. Notice that the spectrum of $P_0(w+S)$ on $H_0$ is $\{e^{\mathbf{i}t};t\in[0,s]\}$, which is out of the unit disk. Then $\lim\nolimits_{n\rightarrow+\infty}(P_0(w+S))^nf=\infty$. Following from the definition of Kalisch operator, one can see $(P_0(w+S))^n=P_0(w+S)^n$. Consequently,
$$
\lim\limits_{n\rightarrow+\infty}\|P_1(w+S)^nf\|\geq\lim\limits_{n\rightarrow+\infty}\|P_0(w+S)^nf\|=+\infty.
$$

{\bf Claim 2.} There exists $M>0$ such that for any $f\in H_2$, $\|(w+S)^nf\|\leq M\|f\|$ for all $n\in \mathbb{N}$.

Given any $f\in H_2$. Let $w=de^{\mathbf{i}\xi_0}$, where $d>0$ and $\xi\in[0,2\pi)$. Then for each $\theta\in[0,2\pi]$,
$$
|w+e^{\mathbf{i}\theta}|=\sqrt{(de^{-\mathbf{i}\xi_0}+e^{-\mathbf{i}\theta})(de^{\mathbf{i}\xi_0}+e^{\mathbf{i}\theta})}=\sqrt{d^2+1+2d\cos(\theta-\xi_0)}.
$$
Following from
$$
\sqrt{d^2+1+2d\cos(a-\xi_0)}=\sqrt{d^2+1+2d\cos(b-\xi_0)}=1,
$$
we have
$$
\cos(a-\xi_0)=\cos(b-\xi_0)=-\frac{d}{2}.
$$
Then, there exist $a<a_0<b_0<b$ such that
$$
-\frac{d+1}{3}\leq\cos(\theta-\xi_0)\leq -\frac{d}{2} \ \  \ \ for \ \theta\in[a,a_0],
$$
and
$$
-\frac{d+1}{3}\leq\cos(\theta-\xi_0)\leq -\frac{d}{2} \ \  \ \ for \ \theta\in[b_0,b].
$$
Moreover,
$$
\sqrt{1-\frac{(d+1)^2}{9}}\leq\sin(\theta-\xi_0)\leq \sqrt{1-\frac{d^2}{4}} \ \  \ \ for \ \theta\in[a,a_0],
$$
and
$$
-\sqrt{1-\frac{d^2}{4}}\leq\sin(\theta-\xi_0)\leq -\sqrt{1-\frac{(d+1)^2}{9}} \ \  \ \ for \ \theta\in[b_0,b].
$$
For $\theta\in[a,a_0]\cup[b_0,b]$, let
$$
G(\theta)=\sqrt{d^2+1+2d\cos(\theta-\xi_0)}
$$
and
$$
g(\theta)=\frac{-d\sin(\theta-\xi_0)}{\sqrt{d^2+1+2d\cos(\theta-\xi_0)}}.
$$
Then $G'(\theta)=g(\theta)$, $|G(\theta)|\leq1$ and
\begin{align*}
-g(\theta)&=\frac{d\sin(\theta-\xi_0)}{\sqrt{d^2+1+2d\cos(\theta-\xi_0)}} \\
&=|\frac{d\sin(\theta-\xi_0)}{\sqrt{d^2+1+2d\cos(\theta-\xi_0)}}| \\
&\geq|d\sin(\theta-\xi_0)| \\
&\geq d\sqrt{1-\frac{(d+1)^2}{9}} \\
&=\frac{d\sqrt{9-(d+1)^2}}{3}.
\end{align*}
Consequently, for every $n\in \mathbb{N}$,
\begin{align*}
&n\int\nolimits_{a}^{a_0}|\mathbf{i}e^{\mathbf{i}t}(w+e^{\mathbf{i}t})^{n-1}|dt  \\
=&\int\nolimits_{a}^{a_0}n(\sqrt{d^2+1+2d\cos(t-\xi_0)})^{n-1}dt \\
=&\int\nolimits_{a}^{a_0}nG(t)^{n-1}\cdot g(t)\cdot\frac{1}{g(t)}dt   \\
\leq& -\frac{3}{d\sqrt{9-(d+1)^2}}\int\nolimits_{a}^{a_0}nG(t)^{n-1}dG(t) \\
=& \frac{3}{d\sqrt{9-(d+1)^2}}((G(a))^{n}-((G(a_0))^{n}) \\
\leq& \frac{3}{d\sqrt{9-(d+1)^2}}.
\end{align*}
For any $f\in H_2$ and any $\theta\in[a,2\pi]$, one can see
\begin{align*}
&|\int\nolimits_{a}^{\theta}n\mathbf{i}e^{\mathbf{i}t}(w+e^{\mathbf{i}t})^{n-1}{\bf 1}_{[a,a_0]}f(t)dt|  \\
\leq&n\int\nolimits_{a}^{\theta}|\mathbf{i}e^{\mathbf{i}t}(w+e^{\mathbf{i}t})^{n-1}{\bf 1}_{[a,a_0]}|dt \cdot(\int\nolimits_{a}^{\theta}|f(t)|^2dt)^{\frac{1}{2}} \\
\leq&n\int\nolimits_{a}^{a_0}|\mathbf{i}e^{\mathbf{i}t}(w+e^{\mathbf{i}t})^{n-1}|dt \cdot(\int\nolimits_{0}^{2\pi}|f(t)|^2dt)^{\frac{1}{2}} \\
\leq& \frac{3}{d\sqrt{9-(d+1)^2}}\|f\|.
\end{align*}
Similarly, we could also prove
$$
n\int\nolimits_{b_0}^{b}|\mathbf{i}e^{\mathbf{i}t}(w+e^{\mathbf{i}t})^{n-1}|dt \leq \frac{3}{d\sqrt{9-(d+1)^2}},
$$
and for any $f\in H_2$ and any $\theta\in[b_0,2\pi]$,
$$
|\int\nolimits_{a}^{\theta}n\mathbf{i}e^{\mathbf{i}t}(w+e^{\mathbf{i}t})^{n-1}{\bf 1}_{[b_0,b]}f(t)dt| \leq \frac{3}{d\sqrt{9-(d+1)^2}}\|f\|.
$$

On the other hand, since for any $\theta\in[a_0,b_0]$,
$$
n|\mathbf{i}e^{\mathbf{i}\theta}(w+e^{\mathbf{i}\theta})^{n-1}|\leq n(\sqrt{d^2+1+\frac{2d(d+1)}{3}})^{n-1}\rightarrow0 \ \ \ \ as \ n\rightarrow+\infty,
$$
there  exists $B>0$ such that for all $n\geq N$,
$$
n\int\nolimits_{a_0}^{b_0}|\mathbf{i}e^{\mathbf{i}t}(w+e^{\mathbf{i}t})^{n-1}|dt \leq B.
$$
Then for any $f\in H_2$ and any $\theta\in[a_0,2\pi]$,
$$
|\int\nolimits_{a}^{\theta}n\mathbf{i}e^{\mathbf{i}t}(w+e^{\mathbf{i}t})^{n-1}{\bf 1}_{[a_0,b_0]}f(t)dt| \leq B\|f\|.
$$

Now, Choose $M=1+\sqrt{2\pi}(B+\frac{6}{d\sqrt{9-(d+1)^2}})$. Thus, for any $f\in H_2$ and all $n\geq N$,
\begin{align*}
&\|(w+S)^{n}f(\theta)\|  \\
\leq& \|(w+e^{\mathbf{i}\theta})^{n}f(\theta)\|+\|\int\nolimits_{a}^{\theta}n\mathbf{i}e^{\mathbf{i}t}(w+e^{\mathbf{i}t})^{n-1}f(t)dt\|  \\
\leq&\|\int\nolimits_{a}^{\theta}n\mathbf{i}e^{\mathbf{i}t}(w+e^{\mathbf{i}t})^{n-1}{\bf 1}_{[a,a_0]}f(t)dt\|+\|\int\nolimits_{a}^{\theta}n\mathbf{i}e^{\mathbf{i}t}(w+e^{\mathbf{i}t})^{n-1}{\bf 1}_{[a_0,b_0]}f(t)dt\|  \\
&+\|\int\nolimits_{a}^{\theta}n\mathbf{i}e^{\mathbf{i}t}(w+e^{\mathbf{i}t})^{n-1}{\bf 1}_{[b_0,b]}f(t)dt\|+\|f\| \\
\leq&\|f\|+\frac{3}{d\sqrt{9-(d+1)^2}}\|f\|\sqrt{2\pi-a}+B\|f\|\sqrt{2\pi-a_0} \\
&+\frac{3}{d\sqrt{9-(d+1)^2}}\|f\|\sqrt{2\pi-b_0} \\
\leq&(1+\sqrt{2\pi}(B+\frac{6}{d\sqrt{9-(d+1)^2}}))\|f\| \\
=&M\|f\|.
\end{align*}

{\bf Claim 3.} For any $f\in H_3$ , if $f(\theta)$ is not constant on $[b,2\pi]$, then $\lim\nolimits_{n\rightarrow\infty}(w+S)^nf=\infty$.

It suffices to prove that for $f\in H_3$, if there exist $A>0$ and an increasing sequence of positive integers $\{n_k\}_{k=1}^{+\infty}$ such that $\|(w+S)^{n_k}f\|\leq A$, then $f(\theta)$ is constant on $[b,2\pi]$. Given arbitrary $\delta\in(b,2\pi)$.  Let
$$
H_4=\{f\in\mathcal{L}^2[0,2\pi]; f(\theta)=0, \ for \ all \ \theta\in[0,\delta]\},
$$
and $P_4:\mathcal{L}^2[0,2\pi]\rightarrow H_4$ be the orthogonal projection. Notice that the spectrum of $P_4(w+S)$ on $H_4$ is $\{e^{\mathbf{i}t};t\in[\delta,2\pi]\}$, which is out of the unit disk. Write $\rho=\|w+e^{\mathbf{i}\delta}\|$. Then for $r$ with $1<r<\rho$, there exists $N>0$ such that for $n\geq N$, $\|(P_4(w+S))^nf\|\geq r^n\|f\|$. Let
$$
C_k=\int\nolimits_{b}^{\delta}\mathbf{i}e^{\mathbf{i}t}(w+S)^{n_k-1}f(t)dt.
$$
Since
$$
P_4(w+S)^{n_k}f=(w+S)^{n_k}P_4f-{\bf 1}_{[\delta,2\pi]}C_k
$$
and
$$
{\bf 1}_{[\delta,2\pi]}C_k=(w+S)^{n_k}\frac{{\bf 1}_{[\delta,2\pi]}C_k}{(w+e^{\mathbf{i}\delta})^{n_k}},
$$
then for $n_k\geq n$
\begin{align*}
r^{n_k}\|P_4f-\frac{{\bf 1}_{[\delta,2\pi]}C_k}{(w+e^{\mathbf{i}\delta})^{n_k}}\|&\leq\|(w+S)^{n_k}(P_4f-\frac{{\bf 1}_{[\delta,2\pi]}C_k}{(w+e^{\mathbf{i}\delta})^{n_k}})\| \\
&=\|(w+S)^{n_k}P_4f-{\bf 1}_{[\delta,2\pi]}C_k\| \\
&=\|P_4(w+S)^{n_k}f\| \\
&\leq A.
\end{align*}
Consequently,
$$
\|P_4f-\frac{{\bf 1}_{[\delta,2\pi]}C_k}{(w+e^{\mathbf{i}\delta})^{n_k}}\|\leq\frac{A}{r^{n_k}}\rightarrow0, \ \ \ as \ k\rightarrow+\infty.
$$
Notice that if $f(\theta)$ is not constant on $[b,2\pi]$,  there exists $\epsilon>0$ such that for any
$C\in \mathbb{C}$,
$$
\|P_4f-{\bf 1}_{[\delta,2\pi]}C\|>\epsilon.
$$
Therefore, $f(\theta)$ is constant on $[\delta,2\pi]$. Furthermore, by the arbitrariness of $\delta$, $f(\theta)$ is constant on $[b,2\pi]$.

Now, let us prove the main result. For any $f\in\mathcal{L}^2[0,2\pi]$, put $f_i=P_if$ for $i=1,2,3$. Then $f=f_1+f_2+f_3$. If $f_1\neq0$, by Claim 1, we obtain $\lim\nolimits_{n\rightarrow+\infty}(w+S)^nf=\infty$ and hence $w+S$ is not Li-Yorke chaotic. if $f_1=0$ and $f_3$ is not constant on $[b,2\pi]$, then by Claim 2 and Claim 3, one can see $\lim\nolimits_{n\rightarrow+\infty}(w+S)^nf=\infty$ and hence $w+S$ is not Li-Yorke chaotic. If $f_1=0$ and $f_3$ is constant on $[b,2\pi]$, by Claim 2 and $(w+S)^nf_3=(w+e^{\mathbf{i}b})^nf_3$, we have $\{\|(w+S)^nf\|\}_{n=0}^{+\infty}$ is bounded and hence $w+S$ is not Li-Yorke chaotic.

Notice that $S$ is Li-Yorke (Devaney) chaotic. Thus after similar discussions for other cases, we obtain $S_{LY}(S)=\{0\}$.
\end{proof}

\end{document}